\documentclass[11pt]{article}
\usepackage{epigamath}

\setpapertype{A4}

\usepackage[english]{babel}

\usepackage[dvips]{graphicx}     

\title{\vspace{0cm}On a theorem of Campana and P\u{a}un}
\titlemark{On a theorem of Campana and P\u{a}un}
\author{\vspace{0cm}Christian Schnell}
\authoraddresses{
\authordata{Christian Schnell}{\firstname{Christian} \lastname{Schnell}\\
\institution{Department of Mathematics, Stony Brook University,
Stony Brook, NY 11794-3651, USA}\\
\email{cschnell@math.stonybrook.edu}}
}
\authormark{Ch. Schnell}
\date{\vspace{-5ex}} 
\journal{\'Epijournal de G\'eom\'etrie Alg\'ebrique} 
\acceptation{Received by the Editors on April 25, 2017, and in final form on June 27, 2017. \\ Accepted on July 24, 2017.}

\newcommand{\articlereference}{Volume 1 (2017), Article Nr. 8}

\acknowledgement{I thank Stefan Kebekus, Mihai P\u{a}un, and Mihnea Popa for their input.
During the preparation of this paper, I have been supported in part by grants
DMS-1404947 and DMS-1551677 from the National Science Foundation.}

\newtheorem{theorem}{Theorem}
\newtheorem{lemma}[theorem]{Lemma}
\newtheorem{remark}[theorem]{Remark}
\newtheorem{proposition}[theorem]{Proposition}

\newcommand{\theoremref}[1]{Theorem \ref{#1}}
\newcommand{\lemmaref}[1]{Lemma \ref{#1}}
\newcommand{\propositionref}[1]{Proposition \ref{#1}}

\usepackage{enumitem}

\usepackage{tikz}
\usepackage{tikz-cd}
\tikzset{commutative diagrams/arrow style=Latin Modern}

\usepackage{mathrsfs}

\newcommand{\shT}{\mathscr{T}}

\newcommand{\tensor}{\otimes}

\newcommand{\RR}{\mathbb{R}}

\newcommand{\PP}{\mathbb{P}}

\DeclareMathOperator{\rk}{rk}

\newcommand{\define}[1]{\emph{#1}}

\newcommand{\shf}[1]{\mathscr{#1}}
\newcommand{\OX}{\shf{O}_X}
\newcommand{\OmX}{\Omega_X}

\newcommand{\restr}[1]{\big\vert_{#1}}

\def\overbar#1#2#3{{%
        \setbox0=\hbox{\displaystyle{#1}}%
        \dimen0=\wd0
        \advance\dimen0 by -#2 
        \vbox {\nointerlineskip \moveright #3 \vbox{\hrule height 0.3pt width \dimen0}%
                \nointerlineskip \vskip 1.5pt \box0}%
}}

\newcommand{\into}{\hookrightarrow}

\newcommand{\pu}{p^{\ast}}

\newcommand{\shF}{\shf{F}}
\newcommand{\shG}{\shf{G}}

\newcommand{\shB}{\mathcal{B}}

\newcommand{\parref}[1]{\hyperref[#1]{\S\ref*{#1}}}
\newcommand{\chapref}[1]{\hyperref[#1]{Chapter~\ref*{#1}}}

\makeatletter
\newcommand*\if@single[3]{%
  \setbox0\hbox{${\mathaccent"0362{#1}}^H$}%
  \setbox2\hbox{${\mathaccent"0362{\kern0pt#1}}^H$}%
  \ifdim\ht0=\ht2 #3\else #2\fi
  }
\newcommand*\rel@kern[1]{\kern#1\dimexpr\macc@kerna}
\newcommand*\widebar[1]{\@ifnextchar^{{\wide@bar{#1}{0}}}{\wide@bar{#1}{1}}}
\newcommand*\wide@bar[2]{\if@single{#1}{\wide@bar@{#1}{#2}{1}}{\wide@bar@{#1}{#2}{2}}}
\newcommand*\wide@bar@[3]{%
  \begingroup
  \def\mathaccent##1##2{%
    \if#32 \let\macc@nucleus\first@char \fi
    \setbox\z@\hbox{$\macc@style{\macc@nucleus}_{}$}%
    \setbox\tw@\hbox{$\macc@style{\macc@nucleus}{}_{}$}%
    \dimen@\wd\tw@
    \advance\dimen@-\wd\z@
    \divide\dimen@ 3
    \@tempdima\wd\tw@
    \advance\@tempdima-\scriptspace
    \divide\@tempdima 10
    \advance\dimen@-\@tempdima
    \ifdim\dimen@>\z@ \dimen@0pt\fi
    \rel@kern{0.6}\kern-\dimen@
    \if#31
      \overline{\rel@kern{-0.6}\kern\dimen@\macc@nucleus\rel@kern{0.4}\kern\dimen@}%
      \advance\dimen@0.4\dimexpr\macc@kerna
      \let\final@kern#2%
      \ifdim\dimen@<\z@ \let\final@kern1\fi
      \if\final@kern1 \kern-\dimen@\fi
    \else
      \overline{\rel@kern{-0.6}\kern\dimen@#1}%
    \fi
  }%
  \macc@depth\@ne
  \let\math@bgroup\@empty \let\math@egroup\macc@set@skewchar
  \mathsurround\z@ \frozen@everymath{\mathgroup\macc@group\relax}%
  \macc@set@skewchar\relax
  \let\mathaccentV\macc@nested@a
  \if#31
    \macc@nested@a\relax111{#1}%
  \else
    \def\gobble@till@marker##1\endmarker{}%
    \futurelet\first@char\gobble@till@marker#1\endmarker
    \ifcat\noexpand\first@char A\else
      \def\first@char{}%
    \fi
    \macc@nested@a\relax111{\first@char}%
  \fi
  \endgroup
}
\makeatother

\newcommand{\shQ}{\mathscr{Q}}
\renewcommand{\shB}{\mathscr{B}}
\newcommand{\shFdel}{\shF_{\Delta}}

\newcommand{\htensor}{\hat{\tensor}}
\newcommand{\shQtor}{\shQ_{\mathit{tor}}}
\newcommand{\Deltahor}{\Delta^{\!\mathit{hor}}}
\newcommand{\Xtl}{\tilde{X}}
\newcommand{\ptl}{\tilde{p}}
\newcommand{\alphatl}{\tilde{\alpha}}
\newcommand{\shFtl}{\tilde{\shF}}
\newcommand{\Deltatl}{\tilde{\Delta}}

\begin{document}


\maketitle



\begin{prelims}


\def\abstractname{Abstract}
\abstract{Let $X$ be a smooth projective variety over the complex numbers, and
$\Delta \subseteq X$ a reduced divisor with normal crossings. We present a slightly
simplified proof for the following theorem of Campana and P\u{a}un: If some tensor
power of the bundle $\OmX^1(\log \Delta)$ contains a subsheaf with big determinant,
then $(X, \Delta)$ is of log general type.  This result is a key step in the recent
proof of Viehweg's hyperbolicity conjecture.}

\keywords{Viehweg's hyperbolicity conjecture; log general type; log cotangent bundle;
foliation; movable curve class; slope semi-stability}

\MSCclass{14E99; 14F10}


\languagesection{Fran\c{c}ais}{%

\textbf{Titre. Sur un th\'eor\`eme de Campana et P\u{a}un}
\commentskip
\textbf{R\'esum\'e.}
Soit $X$ une vari\'et\'e projective complexe lisse et $\Delta\subseteq X$ un diviseur r\'eduit \`a croisements normaux. Nous pr\'esentons une d\'emonstration l\'eg\`erement simplifi\'ee du th\'eor\`eme suivant de Campana et P\u{a}un : si une puissance tensorielle du fibr\'e $\Omega^1_X(\log(\Delta))$ contient un faisceau dont le d\'eterminant est \emph{big}, la paire $(X,\Delta)$ est alors de log-type g\'en\'eral. Ce r\'esultat est une \'etape cl\'e dans la r\'ecente d\'emonstration de la conjecture d'hyperbolicit\'e de Viehweg.}

\end{prelims}


\newpage

\setcounter{tocdepth}{1}
\tableofcontents

\section{Introduction}

The purpose of this paper is to present a slightly simplified proof for the following
result by Campana and P\u{a}un \cite[Theorem~7.6]{CP}. It is a crucial step in the
proof of Viehweg's hyperbolicity conjecture for families of canonically polarized
manifolds \cite[Theorem~7.13]{CP}, and more generally, for smooth families of
varieties of general type \cite[Theorem~A]{families}.

\begin{theorem} \label{thm:CP}
Let $X$ be a smooth projective variety, and $\Delta \subseteq X$ a reduced
divisor with at worst normal crossing singularities. If some tensor power of
$\OmX^1(\log \Delta)$ contains a subsheaf with big determinant, then $K_X + \Delta$
is big.
\end{theorem}

The simplification is that I have substituted an inductive procedure for the
arguments involving Campana's ``orbifold cotangent bundle''; otherwise, the proof of
\theoremref{thm:CP} that I present here is essentially the same as in the one in
\cite{CP}. My reason for writing this paper is that it gives me a chance to draw
attention to some of the beautiful ideas involved in the proof by Campana and
P\u{a}un: slope stability with respect to movable classes; a criterion for the leaves
of a foliation to be algebraic subvarieties; and positivity results for relative
canonical bundles. 

\begin{remark}
{\rm The most recent \texttt{arXiv} version of the paper by Campana and P\u{a}un
(from June 14, 2017) also contains a brief summary of our proof; see
\cite[Section~8.1]{CP}.}
\end{remark}

\section{Strategy of the proof}

Let $(X, \Delta)$ be a pair, consisting of a smooth projective variety $X$ and a
reduced divisor $\Delta \subseteq X$ with at worst normal crossing
singularities. We denote the logarithmic cotangent bundle by the symbol $\OmX^1(\log
\Delta)$, and its dual, the logarithmic tangent bundle, by the symbol $\shT_X(-\log
\Delta)$. Recall that $\shT_X(-\log \Delta)$ is naturally a subsheaf of the tangent
bundle $\shT_X$, and that it is closed under the Lie bracket on $\shT_X$. Indeed,
suppose that $\Delta$ is given, in suitable local coordinates $x_1, x_2, \dotsc,
x_n$, by the equation $x_1 x_2 \dotsm x_k = 0$; then $\shT_X(-\log \Delta)$ is
generated by the $n$ commuting vector fields 
\[
        x_1 \frac{\partial}{\partial x_1}, \dotsc, x_k \frac{\partial}{\partial x_k},
        \frac{\partial}{\partial x_{k+1}}, \dotsc, \frac{\partial}{\partial x_n},
\]
and is therefore closed under the Lie bracket.

Suppose that $\OmX^1(\log \Delta)^{\tensor N}$ contains a subsheaf
with big determinant, for some $N \geq 1$. The following observation reduces the
problem to the case of line bundles.

\begin{lemma} \label{lem:line-bundle}
If $\OmX^1(\log \Delta)^{\tensor N}$ contains a subsheaf of generic rank $r \geq 1$
and with big determinant, then $\OmX^1(\log \Delta)^{\tensor Nr}$ contains a big line
bundle.
\end{lemma}

\begin{proof}
Let $\shB \subseteq \OmX^1(\log \Delta)^{\tensor N}$ be a subsheaf of generic rank $r
\geq 1$, with the property that $\det \shB$ is big. After replacing $\shB$ by its
saturation, whose determinant is of course still big, we may assume that the quotient sheaf 
\[
        \OmX^1(\log \Delta)^{\tensor N} \big/ \shB
\]
is torsion-free, hence locally free outside a closed subvariety $Z \subseteq X$ of
codimension $\geq 2$. On $X \setminus Z$, we have an inclusion of locally free
sheaves
\[
        \det \shB \into \shB^{\tensor r} \into \OmX^1(\log \Delta)^{\tensor Nr},
\]
which remains valid on $X$ by Hartog's theorem.
\hfill $\Box$
\end{proof}

For the purpose of proving \theoremref{thm:CP}, we are therefore allowed to assume
that $\OmX^1(\log \Delta)^{\tensor N}$ contains a big line bundle $L$ as a subsheaf.
Let $\shQ$ denote the quotient sheaf, and consider the resulting short exact sequence
\begin{equation} \label{eq:SES}
        0 \to L \to \OmX^1(\log \Delta)^{\tensor N} \to \shQ \to 0.
\end{equation}
Since $K_X + \Delta$ represents the first Chern class of $\OmX^1(\log \Delta)$, we
obtain 
\[
        N \cdot (\dim X)^{N-1} \cdot (K_X + \Delta) = c_1(L) + c_1(\shQ)
\]
in $N^1(X)_{\RR}$, the $\RR$-linear span of codimension-one cycles modulo numerical
equivalence. By assumption, the class $c_1(L)$ is big; \theoremref{thm:CP} will
therefore be proved if we manage to show that the class $c_1(\shQ)$ is
pseudo-effective. In fact, we are going to prove the following more general result,
which is of course just a special case of \cite[Theorem~7.6 and Theorem~1.2]{CP}.

\begin{theorem} \label{thm:main}
Let $X$ be a smooth projective variety, and $\Delta \subseteq X$ a reduced
divisor with at worst normal crossing singularities. Suppose that some tensor power
of $\OmX^1(\log \Delta)$ contains a subsheaf with big determinant. Then the first
Chern class of \emph{every} quotient sheaf of \emph{every} tensor power of
$\OmX^1(\log \Delta)$ is pseudo-effective.
\end{theorem}


\section{Slopes and foliations}

To simplify the presentation, we will prove \theoremref{thm:main} by
contradiction. Suppose then that, for some integer $N \geq 1$, and for some quotient
sheaf $\shQ$ of $\OmX^1(\log \Delta)^{\tensor N}$, the class $c_1(\shQ)$ was
\emph{not} pseudo-effective. Let $\shQtor \subseteq \shQ$ denote the torsion
subsheaf. Since
\[
        c_1(\shQ) = c_1(\shQtor) + c_1 \bigl( \shQ/\shQtor \bigr),
\]
and since $c_1(\shQtor)$ is effective, we may replace $\shQ$ by $\shQ/\shQtor$, and
assume without any loss of generality that $\shQ$ is torsion-free (and nonzero).

By the characterization of the pseudo-effective cone in \cite[Theorem~2.2]{BDPP},
there is a movable class $\alpha \in N_1(X)_{\RR}$ such that $c_1(\shQ) \cdot
\alpha < 0$. As shown in \cite{CPet,GKP}, there is a good theory of
$\alpha$-semistability for torsion-free sheaves, with almost all the properties that
are familiar from the case of complete intersection curves. We use this theory freely
in what follows. By assumption, 
\[
        \mu_{\alpha}(\shQ) = \frac{c_1(\shQ) \cdot \alpha}{\rk \shQ} < 0,
\]
and so $\shQ$ is a torsion-free quotient sheaf of $\OmX^1(\log \Delta)^{\tensor N}$
with negative $\alpha$-slope. The dual sheaf $\shQ^{\ast}$ is therefore a saturated
subsheaf of $\shT_X(-\log \Delta)^{\tensor N}$ with positive $\alpha$-slope. At this
point, we recall the following result about tensor products.

\begin{theorem} \label{thm:tensor}
Let $\alpha \in N_1(X)_{\RR}$ be a movable class. If $\shF$ and $\shG$ are
torsion-free and $\alpha$-semistable coherent sheaves on $X$, then their tensor product
\[
        \shF \htensor \shG = (\shF \tensor \shG) \big/ (\shF \tensor \shG)_{\mathit{tor}},
\]
modulo torsion, is again $\alpha$-semistable, and $\mu_{\alpha}(\shF \htensor \shG) =
\mu_{\alpha}(\shF) + \mu_{\alpha}(\shG)$.
\end{theorem}

\begin{proof}
For the reflexive hull of the tensor product, this is proved in
\cite[Theorem~4.2 and Proposition~4.4]{GKP}, based on analytic results by Toma
\cite[Appendix]{CPet}.  Since $\shF \htensor \shG$ and its reflexive hull are
isomorphic outside a closed subvariety of codimension $\geq 2$, the assertion
follows. (The formula for the $\alpha$-slope of $\shF \htensor \shG$ is of course
valid for arbitrary nonzero torsion-free coherent sheaves $\shF$ and $\shG$.)
\hfill $\Box$
\end{proof}

Similarly, the fact that $\shT_X(-\log \Delta)^{\tensor N}$
has a subsheaf with positive $\alpha$-slope implies, again by
\cite[Theorem~4.2 and Proposition~4.4]{GKP}, that $\shT_X(-\log \Delta)$ must also
contain a subsheaf with positive $\alpha$-slope. Let $\shFdel \subseteq \shT_X(-\log
\Delta)$ be the maximal $\alpha$-destabilizing subsheaf \cite[Corollary~2.24]{GKP}. 

\begin{lemma} \label{lem:shFdel}
$\shFdel$ is a saturated, $\alpha$-semistable subsheaf of $\shT_X(-\log \Delta)$, 
of positive $\alpha$-slope. Every subsheaf of
$\shT_X(-\log \Delta)/\shFdel$ has $\alpha$-slope less than $\mu_{\alpha}(\shFdel)$.
\end{lemma}

\begin{proof}
This is clear from the construction of the maximal destabilizing subsheaf in
\cite[Corollary~2.4]{GKP}. Note that $\shFdel$ is the first step in the
Harder-Narasimhan filtration of $\shT_X(-\log \Delta)$, see
\cite[Corollary~2.26]{GKP}.
\hfill $\Box$
\end{proof}

Recall that we have an inclusion $\shT_X(-\log \Delta) \subseteq \shT_X$. We define
another coherent subsheaf $\shF \subseteq \shT_X$ as the saturation of $\shFdel$ in
$\shT_X$; then $\shT_X/\shF$ is torsion-free, and
\begin{equation} \label{eq:saturation}
        \shF \cap \shT_X(-\log \Delta) = \shFdel.
\end{equation}
We will see in a moment that $\shF$ is actually a (typically, singular) foliation on
$X$. Recall that, in general, a \define{foliation} on a smooth projective
variety is a saturated subsheaf $\shF \subseteq \shT_X$ that is closed under the Lie
bracket on $\shT_X$. From the Lie bracket, one constructs an $\OX$-linear mapping
\[
        N \colon \shF \htensor \shF \to \shT_X / \shF,
\]
called the \define{O'Neil tensor} of $\shF$; evidently, $\shF$ is a foliation if and
only if its O'Neil tensor vanishes.

\begin{lemma} \label{lem:foliation}
The O'Neil tensor
\[      
        N \colon \shF \htensor \shF \to \shT_X / \shF
\]
vanishes, and $\shF$ is therefore a foliation on $X$.
\end{lemma}

\begin{proof}
The Lie bracket of two sections of $\shT_X(-\log \Delta)$ is a
section of $\shT_X(-\log \Delta)$, and so we get a logarithmic O'Neil tensor
\[
        N_{\Delta} \colon \shFdel \htensor \shFdel \to \shT_X(-\log \Delta)/\shFdel.
\]
The key point is that $N_{\Delta} = 0$. Indeed, by \theoremref{thm:tensor}, the
tensor product $\shFdel \htensor \shFdel$, modulo torsion, is again
$\alpha$-semistable of slope 
\[
        \mu_{\alpha}(\shFdel \htensor \shFdel) = 2 \cdot \mu_{\alpha}(\shFdel) 
                > \mu_{\alpha}(\shFdel),
\]
which is strictly greater than the slope of any nonzero subsheaf of
$\shT_X(-\log \Delta)/\shFdel$ by \lemmaref{lem:shFdel}. This inequality among
slopes implies that $N_{\Delta} = 0$, see for instance \cite[Proposition~2.16 and
Corollary~2.17]{GKP}.

The O'Neil tensor $N$ and the logarithmic O'Neil tensor $N_{\Delta}$ are both induced
by the Lie bracket on $\shT_X$, and so we have the following commutative diagram:
\[
\begin{tikzcd}
\shFdel \htensor \shFdel \dar \rar{N_{\Delta}} & \shT_X(-\log \Delta) / \shFdel \dar[hook] \\
\shF \htensor \shF \rar{N} & \shT_X/\shF
\end{tikzcd}
\]
The vertical arrow on the right is injective by \eqref{eq:saturation}. Now
$N_{\Delta} = 0$ implies that $N$ factors through the cokernel of the vertical
arrow on the left; but the cokernel is a torsion sheaf, whereas $\shT_X/\shF$ is
torsion-free. The conclusion is that $N = 0$.
\hfill $\Box$
\end{proof}

The next step in the proof is to show that the foliation $\shF$ is actually
algebraic. This is a simple consequence of the powerful algebraicity theorem of
Campana and P\u{a}un \cite[Theorem~1.1]{CP}, which generalizes a well-known result by
Bogomolov and McQuillan \cite{BMQ} and Bost \cite[\S3.3]{Bost} from complete intersection
curves to movable classes. (See also the paper \cite{KST} by Kebekus, Sol\`a Conde,
and Toma.)

\begin{theorem} \label{thm:algebraic}
Let $X$ be a smooth projective variety over the complex numbers, and let $\shF
\subseteq \shT_X$ be a foliation. Suppose that there exists a
movable class $\alpha \in N_1(X)_{\RR}$, such that every nonzero quotient sheaf
of $\shF$ has positive $\alpha$-slope. Then $\shF$ is an algebraic foliation, and its
leaves are rationally connected.
\end{theorem}

To apply this in our setting, we observe that every quotient sheaf of $\shF$ is, at
least over the open subset $X \setminus \Delta$, also a quotient sheaf of $\shFdel$,
because $\shF$ and $\shFdel$ agree outside the divisor $\Delta$. As $\shFdel$ is
$\alpha$-semistable with $\mu_{\alpha}(\shF) > 0$, it follows easily that every
quotient sheaf of $\shF$ has positive $\alpha$-slope. We can now invoke
\theoremref{thm:algebraic} and conclude that the foliation $\shF$ is algebraic. In
other words \cite[\S4]{CP}, there exists a dominant rational mapping
\[
        p \colon X \dashrightarrow Z
\]
to a smooth projective variety $Z$, such that 
\[
        \shF = \ker \bigl( \mathit{dp} \colon \shT_X \to \pu \shT_Z \bigr) 
\]
outside a subset of codimension $\geq 2$. More precisely, let us follow
\cite[Construction~2.29]{CKT} and denote by the symbol $\shT_{X/Z}$ the unique
reflexive sheaf on $X$ that agrees with $\ker \bigl( \mathit{dp} \colon \shT_X \to
\pu \shT_Z \bigr)$ on the big open subset where $p$ is a morphism. Using this
notation, the algebraicity of $\shF$ may be expressed as
\begin{equation} \label{eq:algebraic}
        \shF = \shT_{X/Z};
\end{equation}
indeed, $\shF$ is reflexive, due to the fact that $\shT_X/\shF$ is torsion-free.

\begin{remark}
{\rm \theoremref{thm:algebraic} also says that the fibers of $p$ are rationally connected,
but we are not going to make any use of this extra information. This means that the
proof of \theoremref{thm:main} only uses characteristic zero methods.}
\end{remark}

\section{Pseudo-effectivity}

Let us first convince ourselves that $Z$ cannot be a point. This will later allow us
to argue by induction on the dimension, because the general fiber of $p$ has
dimension less than $\dim X$.

\begin{lemma} \label{lem:dimZ}
With notation as above, we must have $\dim Z \geq 1$.
\end{lemma}

\begin{proof}
If $\dim Z = 0$, then $\shF = \shT_X$ and $\shFdel = \shT_X(-\log \Delta)$, and
consequently, the logarithmic tangent bundle $\shT_X(-\log \Delta)$ is
$\alpha$-semistable of positive slope. Since the tensor product of
$\alpha$-semistable sheaves remains $\alpha$-semistable \cite[Proposition~4.4]{GKP},
this means that any tensor power of $\OmX^1(\log \Delta)$ is $\alpha$-semistable of
negative slope. But that contradicts the hypothesis of \theoremref{thm:main}, namely
that some tensor power of $\OmX^1(\log \Delta)$ contains a subsheaf with big
determinant, because the $\alpha$-slope of such a subsheaf is obviously positive.
\hfill $\Box$
\end{proof}

The only properties of $\shFdel$ that we are still going to use in the proof of
\theoremref{thm:main} are the identity in \eqref{eq:saturation}, and the fact that
$c_1(\shFdel) \cdot \alpha > 0$ for a movable class $\alpha \in N_1(X)_{\RR}$. In
return, we are allowed to assume that $p \colon X \to Z$ is a morphism.

\begin{lemma} \label{lem:morphism}
Without loss of generality, $p \colon X \to Z$ is a morphism.
\end{lemma}

\begin{proof}
Choose a birational morphism $f \colon \Xtl \to X$, for example by resolving the
singularities of the closure of the graph of $p \colon X \dashrightarrow Z$
inside $X \times Z$, with the following properties: the rational mapping $p \circ f$
extends to a morphism $\ptl \colon \Xtl \to Z$; both $K_{\Xtl/X}$ and $\ptl^{\ast}
\Delta$ are normal crossing divisors; and $f$ is an isomorphism over the open subset
where $p$ is already a morphism. 

Let $\Deltatl$ be the reduced normal crossing divisor whose support is equal to the
preimage of $\Delta$ in $\Xtl$. Then
\[
        \Omega_{\Xtl}^1(\log \Deltatl) \cong \ptl^{\ast} \OmX^1(\log \Delta),
\]
and since the pullback of a big line bundle by $\ptl$ stays big, it is still true
that some tensor power of $\Omega_{\Xtl}^1(\log \Deltatl)$ contains a big line
bundle as a subsheaf. Now define
\[
        \shFtl = \shT_{\Xtl/Z} = 
                \ker \bigl( \ptl^{\ast} \colon \shT_{\Xtl} \to \ptl^{\ast} \shT_Z \bigr),
\]
which is a saturated subsheaf of $\shT_{\Xtl}$. The intersection
\[
        \shFtl \cap \shT_{\Xtl}(-\log \Deltatl)
\]
is a saturated (and hence reflexive) subsheaf of $\shT_{\Xtl}(-\log \Deltatl)$, whose
pushforward to $X$ is isomorphic to $\shFdel$, by \eqref{eq:saturation} and the fact
that $\shFdel$ is reflexive. Consequently,
\[
        c_1 \Bigl( \shFtl \cap \shT_{\Xtl}(-\log \Deltatl) \Bigr) \cdot \alphatl
                = c_1(\shFdel) \cdot \alpha > 0,
\]
where the class $\alphatl = \ptl^{\ast} \alpha \in N_1(\Xtl)_{\RR}$ is of course still
movable. Nothing essential is therefore changed if we replace the rational mapping $p
\colon X \dashrightarrow Z$ by the morphism $\ptl \colon \Xtl \to Z$; the divisor
$\Delta \subseteq X$ by $\Deltatl \subseteq \Xtl$; the sheaf
$\shFdel$ by the intersection
\[
        \shT_{\Xtl/Z} \cap \shT_{\Xtl}(-\log \Deltatl) \subseteq \shT_{\Xtl}
\]
and the movable class $\alpha \in N_1(X)_{\RR}$ by its pullback $\alphatl =
\ptl^{\ast} \alpha$.
\hfill $\Box$
\end{proof}

Let $R(p)$ denote the ramification divisor of the morphism $p \colon X \to Z$; see
\cite[Definition~2.16]{CKT} for the precise definition. Recall from
\cite[Lemma~2.31]{CKT} the following formula for the first Chern class of our
foliation $\shF \subseteq \shT_X$, in $N^1(X)_{\RR}$:
\begin{equation} \label{eq:CKT}
c_1(\shF) = c_1(\shT_{X/Z}) = - K_{X/Z} + R(p)
\end{equation}
Computing the first Chern class of $\shFdel$ is a little tricky
\cite[Proposition~5.1]{CP}, but at least we can use the fact that $\shF =
\shT_{X/Z}$ to estimate the difference 
\[
        c_1(\shF) - c_1(\shFdel) = c_1(\shF/\shFdel).
\]
Recall that the \define{horizontal part} $\Deltahor \subseteq \Delta$ is the union of
all irreducible components of $\Delta$ that map onto $Z$; evidently, $\Deltahor$ is
again a reduced divisor on $X$ with at worst normal crossing singularities.

\begin{lemma} \label{lem:difference}
The class $c_1(\shF) - c_1(\shFdel) - \Deltahor$ is effective.
\end{lemma}

\begin{proof}
It is easy to see from \eqref{eq:saturation} that we have an inclusion of sheaves
\[
        \shF/\shFdel \into \shT_X \big/ \shT_X(-\log \Delta).
\]
The sheaf on the right-hand side is supported on the divisor $\Delta$, and a brief
computation shows that
\[
        \shT_X \big/ \shT_X(-\log \Delta) \cong
                \bigoplus_{D \subseteq \Delta} \mathscr{N}_{D|X}
\]
is isomorphic to the direct sum of the normal bundles of the irreducible components
of $\Delta$. The rank of $\shF/\shFdel$ at the generic point of $D$ is thus either $0$
or $1$, and 
\[
        c_1(\shF/\shFdel) = \sum_{D \subseteq \Delta} a_D D,
\]
where $a_D = 0$ if $\shF = \shFdel$ at the generic point of $D$, and $a_D = 1$
otherwise. To prove that $c_1(\shF/\shFdel) - \Deltahor$ is effective, we only have
to argue that $\shF \neq \shFdel$ at the generic point of
each irreducible component of $\Deltahor$. This is a consequence of the fact that
$\shF = \shT_{X/Z}$, as we now explain.

Fix an irreducible component $D$ of the horizontal part $\Deltahor$. At the generic
point of $D$, the morphism $p \colon X \to Z$ is smooth. After choosing suitable
local coordinates $x_1, \dotsc, x_n$ in a neighborhood of a sufficiently general
point of $D$, we may therefore assume that $p$
is locally given by
\[
        p(x_1, \dotsc, x_n) = (x_1, \dotsc, x_d),
\]
where $d = \dim Z$, and that the divisor $\Delta$ is defined by the equation $x_n =
0$. In these local coordinates, $\shF = \shT_{X/Z}$ is the subbundle of $\shT_X$ spanned by
\begin{align*}
        \frac{\partial}{\partial x_n}, &\frac{\partial}{\partial x_{n-1}}, \dotsc,
                \frac{\partial}{\partial x_{d+1}}. \\
\intertext{On the other hand, the subsheaf $\shT_X(-\log \Delta)$ is spanned by the vector
fields}
        x_n \frac{\partial}{\partial x_n}, &\frac{\partial}{\partial x_{n-1}}, \dotsc,
                \frac{\partial}{\partial x_{d+1}}, \dotsc, \frac{\partial}{\partial x_1},
\end{align*}
and so it is clear from \eqref{eq:saturation} that $\shF \neq \shFdel$ in a
neighborhood of the given point.
\hfill $\Box$
\end{proof}

From \lemmaref{lem:difference}, we draw the conclusion that
\begin{equation} \label{eq:shFdel}
        - \bigl( K_{X/Z} + \Deltahor - R(p) \bigr) \cdot \alpha 
                = \bigl( c_1(\shF) - \Deltahor \bigr) \cdot \alpha
                \geq c_1(\shFdel) \cdot \alpha > 0,
\end{equation}
where $\alpha \in N_1(X)_{\RR}$ is the movable class from above. We will therefore
reach the desired contradiction if we manage to prove that the divisor class $K_{X/Z}
+ \Deltahor - R(p)$ is pseudo-effective. According to \cite[Theorem~3.3]{CP} or to
\cite[Theorem~7.1]{CKT}, it is actually enough to check that $K_F + \Delta_F$ is
pseudo-effective for a general fiber $F$ of the morphism $p$; and we can prove, by
induction on the dimension, that $K_F + \Delta_F$ is not only pseudo-effective, but
even big. The results that we use here are slight improvements of
\cite[Theorem~4.13]{Campana}, which is itself a generalization of Viehweg's weak
positivity theorem.

\section{Induction on the dimension}

In this section, we use induction on the dimension to finish the proof of
\theoremref{thm:main} and \theoremref{thm:CP}. 

\begin{proposition} \label{prop:fiber}
Suppose that \theoremref{thm:CP} is true in dimension less than $\dim X$. If some
tensor power of $\OmX^1(\log \Delta)$ contains a subsheaf with big determinant, then
$K_{X/Z} + \Deltahor$ is pseudo-effective.
\end{proposition}

\begin{proof}
Let $F$ be a general fiber of the morphism $p \colon X \to Z$; since $\dim Z \geq 1$,
we have $\dim F \leq \dim X - 1$. Denote by $\Delta_F$ the restriction of $\Delta$;
since $F$ is a general fiber, $\Delta_F$ is still a normal crossing divisor. Clearly
\[
        (K_{X/Z} + \Deltahor) \restr{F} = K_F + \Delta_F,
\]
and according to \cite[Theorem~7.3]{CKT}, the pseudo-effectivity of $K_{X/Z} +
\Deltahor$ will follow if we manage to show that $K_F + \Delta_F$ is pseudo-effective.

By hypothesis and by \lemmaref{lem:line-bundle}, there is a nonzero morphism 
\[
        L \to \OmX^1(\log \Delta)^{\tensor k}
\]
from a big line bundle $L$ to some tensor power of $\OmX^1(\log \Delta)$. Since $F$
is a general fiber of $p \colon X \to Z$, we can restrict this morphism to $F$ to
obtain a nonzero morphism
\[
        L_F \to \Bigl( \OmX^1(\log \Delta) \restr{F} \Bigr)^{\tensor k}.
\]
Here $L_F$ denotes the restriction of $L$ to the fiber; since $L$ is big, $L_F$ is
also big. 

The inclusion of $F$ into $X$ gives rise to a short exact sequence
\[
        0 \to \shf{N}_{F|X} \to \OmX^1(\log \Delta) \restr{F} \to
                \Omega_F^1(\log \Delta_F) \to 0,
\]
which induces a filtration on the $k$-th tensor power of the locally free sheaf in the
middle. Since the normal bundle $\shf{N}_{F|X}$ is trivial of rank $\dim Z$, we find,
by looking at the subquotients of this filtration, that there is a nonzero morphism
\[
        L_F \to \Omega_F^1(\log \Delta_F)^{\tensor j}
\]
for some $0 \leq j \leq k$. Because $L_F$ is big, we actually have $1 \leq j \leq k$.
Since we are assuming that \theoremref{thm:CP} is true for the pair $(F, \Delta_F)$,
the class $K_F + \Delta_F$ is big on $F$, hence pseudo-effective. Appealing to
\cite[Theorem~7.3]{CKT}, we deduce that the class $K_{X/Z} + \Deltahor$ is
pseudo-effective on~$X$.~\hfill $\Box$
\end{proof}

By induction on the dimension, the two assumptions of \propositionref{prop:fiber} are
met in our case, and the class $K_{X/Z} + \Deltahor$ is therefore pseudo-effective.
According to \cite[Theorem~7.1]{CKT}, this implies that $K_{X/Z} + \Deltahor -
R(p)$ is also pseudo-effective.%
\footnote{As stated, both \cite[Theorem~3.3]{CP} and \cite[Theorem~7.1]{CKT} actually
assume that $K_X + \Delta$ is
pseudo-effective, but in the case of a morphism $p \colon X \to Z$, the proofs go
through under the weaker hypothesis that $K_{X/Z} + \Deltahor$ is pseudo-effective.}
Going back to the inequality in \eqref{eq:shFdel}, we find that
\[
        0 \geq -\bigl( K_{X/Z} + \Deltahor      - R(p) \bigr) \cdot \alpha \geq
        c_1(\shFdel) \cdot \alpha > 0,
\]
and so we have reached the desired contradiction. The conclusion is that $c_1(\shQ)$
is indeed pseudo-effective, and so \theoremref{thm:main} and \theoremref{thm:CP} are
proved.

\begin{remark}
{\rm Most of the argument, for example the proof of \lemmaref{lem:dimZ}, goes through
when some tensor power of $\OmX^1(\log \Delta)$ contains a
subsheaf with \emph{pseudo-effective} determinant. But \theoremref{thm:main} is
obviously not true under this weaker hypothesis: for example, on the product $E
\times \PP^1$ of an elliptic curve and $\PP^1$, there are nontrivial one-forms, yet
the canonical bundle is not pseudo-effective. What happens is that the last step in
the proof of \propositionref{prop:fiber} breaks down: when $L$ is not big, it may be
that $j=0$ (and $L_F$ is then trivial).}
\end{remark}

\providecommand{\bysame}{\leavevmode\hbox to3em{\hrulefill}\thinspace}
%
%

\bibliographystyle{amsalpha}

\begin{thebibliography}{BDPP13}
\bibitem[BDPP13]{BDPP}
S\'ebastien Boucksom, Jean-Pierre Demailly, Mihai P\u{a}un, and Thomas
  Peternell, \emph{The pseudo-effective cone of a compact {K}\"ahler manifold
  and varieties of negative {K}odaira dimension}, J. Algebraic Geom.
  \textbf{22} (2013), no.~2, 201--248. 
\MR{3019449}

\bibitem[BM16]{BMQ}
Fedor Bogomolov and Michael McQuillan, \emph{Rational curves on foliated
  varieties}. In: Foliation Theory in Algebraic Geometry (Paolo Cascini, James
  McKernan, and Jorge~Vit{\'o}rio Pereira, eds.),  pp.~21--51, Springer International
  Publishing, Cham, 2016.
 \href{http://preprints.ihes.fr/M01/M01-07.ps.gz}{ihes/M01-07}

\bibitem[Bos01]{Bost}
Jean-Beno{\^\i}t Bost, \emph{Algebraic leaves of algebraic foliations over
  number fields}, Publ. Math. Inst. Hautes \'Etudes Sci. (2001), no.~93,
  161--221.
\MR{1863738}

\bibitem[Cam04]{Campana}
Fr\'ed\'eric Campana, \emph{Orbifolds, special varieties and classification
  theory}, Ann. Inst. Fourier (Grenoble) \textbf{54} (2004), no.~3, 499--630.
\MR{2097416}

\bibitem[CKT16]{CKT}
Beno{\^\i}t Claudon, Stefan Kebekus, and Behrouz Taji, \emph{Generic positivity
  and applications to hyperbolicity of moduli spaces}, preprint 2016.
\href{https://arxiv.org/abs/1610.09832}{arXiv:1610.09832}

\bibitem[CP11]{CPet}
Fr\'ed\'eric Campana and Thomas Peternell, \emph{Geometric stability of the
  cotangent bundle and the universal cover of a projective manifold}, with an appendix by Matei Toma, Bull.
  Soc. Math. France \textbf{139} (2011), no.~1, 41--74.
\MR{2815027} 

\bibitem[CP15]{CP}
Fr\'ed\'eric Campana and Mihai P\u{a}un, \emph{Foliations with positive slopes
  and birational stability of orbifold cotangent bundles}, preprint 2015.
\href{https://arxiv.org/abs/1508.02456}{arXiv:1508.02456}

\bibitem[GKP16]{GKP}
Daniel Greb, Stefan Kebekus, and Thomas Peternell, \emph{Movable curves and
  semistable sheaves}, Int. Math. Res. Not. IMRN (2016), no.~2, 536--570.
\MR{3493425}

\bibitem[KST07]{KST}
Stefan Kebekus, Luis Sol\'a Conde, and Matei Toma, \emph{Rationally connected
foliations after Bogomolov and McQuillan}, J. Algebraic Geom. \textbf{16} (2007),
no.~1, 65--81. 
\MR{2257320}

\bibitem[PS17]{families}
Mihnea Popa and {\relax Ch}ristian Schnell, \emph{Viehweg's hyperbolicity
  conjecture for families with maximal variation}, Invent. Math. \textbf{208}
  (2017), no.~3, 677--713.
\MR{3648973}
\end{thebibliography}
\bibliographymark{References}

\end{document}